\documentclass[11pt]{amsart}
\usepackage{epsfig, graphics, psfrag, color}
\usepackage{amsmath, amsfonts, amscd, amssymb, graphicx, mathrsfs,
  url, setspace}

\title{Geodesic systems of tunnels in hyperbolic 3--manifolds}

\author{Stephan D. Burton}
\address{Department of Mathematics, Michigan State University, East
Lansing, MI 48824, USA}
\email{burtons8@msu.edu}

\author{Jessica S. Purcell}
\address{ Department of Mathematics, Brigham Young University,
Provo, UT 84602, USA}
\email{jpurcell@math.byu.edu}

\newcommand{\bdy}{{\partial}} 

\renewcommand{\inf}[1]{{\mathrm{inf} \left\{ #1 \right\} }}

\renewcommand{\phi}{{\varphi}}

\newcommand{\HH}{{\mathbb{H}}}
\newcommand{\RR}{{\mathbb{R}}}
\newcommand{\ZZ}{{\mathbb{Z}}}

\newcommand{\CC}{{\mathbb{C}}}
\newcommand{\PSL}{{\mathrm{PSL}}}

\numberwithin{equation}{section}

\newcommand{\DD}{\mathbf{D}}
\newcommand{\inv}{^{-1}}

\def\co{\colon\thinspace}

\newcommand{\F}{{\mathcal{F}}}

\theoremstyle{plain}
\newtheorem{theorem}{Theorem}[section]

\newtheorem{lemma}[theorem]{Lemma}
\newtheorem{prop}[theorem]{Proposition}

\newtheorem{question}[theorem]{Question}

\newtheorem*{namedtheorem}{\theoremname}
\newcommand{\theoremname}{testing}

\theoremstyle{definition}
\newtheorem{define}[theorem]{Definition}
\newtheorem{definition}[theorem]{Definition}

\newtheorem{example}[theorem]{Example}

\begin{document}

\begin{abstract}
It is unknown whether an unknotting tunnel is always isotopic to a
geodesic in a finite volume hyperbolic 3--manifold.  In this paper,
we address the generalization of this question to hyperbolic
3--manifolds admitting tunnel systems.  We show that there exist
finite volume hyperbolic 3--manifolds with a single cusp, with a
system of $n$ tunnels, $n-1$ of which come arbitrarily close to
self--intersecting.  This gives evidence that systems of unknotting
tunnels may not be isotopic to geodesics in tunnel number $n$
manifolds.  In order to show this result, we prove there is a
geometrically finite hyperbolic structure on a $(1;n)$--compression
body with a system of $n$ core tunnels, $n-1$ of which
self--intersect.  
\end{abstract}

\maketitle


\newcommand{\mat}[2][cccc]{\left(\begin{array}{#1} #2\\
	\end{array}\right)}

\section{Introduction}
One major task in the study of 3--manifolds is to relate topological
invariants to geometric ones, for example, to identify arcs with a
topological description as geodesics in a hyperbolic manifold.  One
particular class of arcs that has earned interest is that of
unknotting tunnels.

An \emph{unknotting tunnel} in a 3--manifold $M$ with torus boundary
is an embedded arc with endpoints on $\partial M$ whose complement is
homeomorphic to a handlebody.  A manifold that admits a single
unknotting tunnel (and is not a solid torus) is called a \emph{tunnel
  number 1} manifold.  A \emph{system of unknotting tunnels} for a
3--manifold $M$ with torus boundary is a collection of arcs $\tau_1,
\dots, \tau_n$, each with endpoints on $\partial M$, such that
$M\setminus (\bigcup_{i=1}^n N(\tau_i))$ is a handlebody, where
$N(\cdot)$ denotes a regular neighborhood.  Manifolds that admit a
tunnel system of $n$ arcs, but not a tunnel system of $(n-1)$ arcs are
called \emph{tunnel number $n$}.  Recall that every 3--manifold with
torus boundary is tunnel number $n$ for some $n$, because the tunnel
number of the manifold encodes the genus of a minimal genus Heegaard
splitting, and every 3--manifold admits a Heegaard splitting.

Now, unknotting tunnels are defined by topology; they are described by
embedded arcs and homeomorphisms.  Adams was the first to investigate
their geometry in the case that the 3--manifold is hyperbolic
\cite{Adams1}.  He showed that when $M$ is a tunnel number 1 manifold
with two boundary components, then an unknotting tunnel will always be
isotopic to a geodesic.  He asked if this is true for more general
tunnel number 1 manifolds.  Soon after, Adams and Reid showed that an
unknotting tunnel in a 2--bridge knot complement is always isotopic to
a geodesic \cite{AdamsReid}.  Recently, Cooper, Futer, and Purcell
\cite{CooperFuterPurcell} showed that in an appropriate sense,
an unknotting tunnel in a tunnel number 1 manifold is isotopic to a
geodesic generically.

In this paper, we investigate the generalization of Adams' question to
systems of unknotting tunnels, or tunnel systems, in tunnel number $n$
manifolds, and give evidence that in this setting, the answer to the
question is no.  That is, we show that there are tunnel number $n$
manifolds, and a system of $n$ tunnels, such that $(n-1)$ of those
tunnels are homotopic to geodesics arbitrarily close to having
self--intersections.  This is the content of Theorem
\ref{thm:maintheorem}, below.  Because the geodesics homotopic to
these tunnels come within distance $\epsilon$ of self--intersecting,
they either must pass through themselves in an attempted isotopy to
the tunnel system, or they lie within distance $\epsilon$ of homotopic
arcs which pass through themselves under the natural homotopy to the
tunnel system.  Thus it seems unlikely that all such tunnels will be
isotopic to geodesics.  Hence this result gives evidence that not all
tunnels of all possible tunnel systems for a tunnel number $n$
manifold will be isotopic to geodesics.  See below for further
discussion.

In order to understand the geometry of tunnel number $n$ manifolds, we
study the geometry of compression bodies with genus 1 negative
boundary, and genus $(n+1)$ positive boundary.  We denote these
compression bodies as \emph{(1; n+1)--compression bodies}.  Notice
that any tunnel number $n$ manifold with a single torus boundary
component is obtained by attaching a genus $(n+1)$ handlebody to a
$(1; n+1)$--compression body along their common genus $(n+1)$
boundaries.  A system of unknotting tunnels in the resulting manifold
will consist of a system of arcs in the $(1; n+1)$--compression body,
which we call \emph{core tunnels}.  In the case of the
$(1;2)$--compression body, Lackenby and Purcell investigated the
natural extension of Adams' question to geometrically finite
hyperbolic structures on such compression bodies
\cite{LackenbyPurcell}.  They conjectured that in the
$(1;2)$--compression body, core tunnels are always isotopic to
geodesics.

Another main result of this paper is that the natural generalization
of Lackenby and Purcell's conjecture to $(1; 1+n)$--compression bodies
is false.

\begin{theorem}\label{thm:main-compress}
There exist geometrically finite hyperbolic structures on the
$(1;~n+1)$--compression body, for $n\geq 2$, for which $(n-1)$ of the
$n$ core tunnels are homotopic to self--intersecting geodesics.  Hence
these tunnels cannot be isotopic to simple geodesics.
\end{theorem}

Theorem \ref{thm:main-compress} is obtained by studying Ford domains
of geometrically finite structures on such compression bodies.  Ford
domains have proven useful for the study of geometrically finite
structures on manifolds in the past.  For example, J{\o}rgensen
studied Ford domains of once punctured torus groups
\cite{JorgensenPuncture} and cyclic groups \cite{JorgensenCyclic}.
Akiyoshi, Sakuma, Wada, and Yamashita extended J{\o}rgensen's work
\cite{Aswy}, and Wada \cite{Wada} developed an algorithm to determine
Ford domains of these manifolds.  Lackenby and Purcell also studied
Ford domains on $(1;2)$--compression bodies \cite{LackenbyPurcell},
and Ford domains play a role in identifying long tunnels in the work
of Cooper, Lackenby, and Purcell \cite{CLP:LengthPaper}.

Using the Ford domains for geometrically finite hyperbolic structures
on $(1; n+1)$ compression bodies, as well as geometric techniques to
attach handlebodies to such structures as in \cite{CLP:LengthPaper},
we show the following.

\begin{theorem}\label{thm:maintheorem}
For any $\epsilon>0$, there exist finite volume one--cusped hyperbolic
manifolds with a system of $n$ tunnels for which $(n-1)$ of the
tunnels are homotopic to geodesics which come within distance
$\epsilon$ of self--intersecting.
\end{theorem}


The proof of this theorem does not guarantee that the geodesics will
self--intersect.  However, the proof involves constructing a sequence
of hyperbolic manifolds with geodesics that are close to
self--intersecting.  In particular, we start with a self--intersecting
geodesic and modify the geometry slightly to obtain the new hyperbolic
manifold.  If the geodesic does not remain self--intersecting under
the geometric modification, then it will move in one of two
directions, only one of which is in the direction of isotopy of the
tunnel.  In the other direction, the obvious homotopy to the
unknotting tunnel passes through the point of self--intersection, and
so is not an isotopy.  (In fact, there may still be a non-obvious
isotopy even in this case, but the homotopy moving the arc the
shortest distance in an $\epsilon$--ball about the nearest points on
the geodesic will pass through the geodesic, so it will not be an
isotopy.)  In any case, the geodesic in the homotopy class of the
tunnel lies within distance $\epsilon$ of an arc which must pass
through a self--intersection in a natural homotopy to the unknotting
tunnel.  Hence we say these tunnels are ``within $\epsilon$'' of not
being isotopic to geodesics.  This gives evidence that these tunnels
are not isotopic to a geodesic, although not a proof of the fact.
Moreover, as there are many choices involved in the proof of Theorem
\ref{thm:maintheorem}, it is plausible that some choice will produce a
hyperbolic manifold with a tunnel system which may have to pass
through itself when homotoped to a nearly self--intersecting geodesic.
Consequently, it is likely that there are finite volume tunnel number
$n$ manifolds for which $(n-1)$ of the tunnels are not isotopic to a
geodesic.

Finally, we note that the proof of this theorem relies upon a specific
choice of the spine of a compression body $C$.  However, there are
countably many choices for any tunnel system for tunnel number $n$
manifolds, provided $n\geq 2$.  In fact, we will see below that our
choice of tunnel systems for Theorem \ref{thm:maintheorem} is not a
natural choice for the geometric structure we start with.  In each of
our examples, there is a more obvious choice of tunnel systems from a
geometric point of view, which leads to a geodesic tunnel system.
Therefore, we ask the following.

\begin{question}\label{quest:tunnelsystems}
For any finite volume tunnel number $n$ manifold with a single cusp,
is there a choice of a system of unknotting tunnels for the manifold
such that each tunnel is isotopic to a geodesic?
\end{question}

\subsection{Acknowledgments}
We thank David Futer, Yair Minsky, Yoav Moriah, and Saul Schleimer for
helpful conversations.  We particularly thank Schleimer for explaining
to us the handlebody complex and its application in
Section \ref{sec:TunnelNumbern}.  Purcell was supported by NSF grant
DMS--1007437 and the Alfred P.~Sloan Foundation.

\section{Background}
In this section, we define notation and review background material on
Ford domains of compression bodies.  We also prove a few lemmas that
will be important later in the paper.  

\subsection{The topology of compression bodies}
Here we review topological facts concerning compression bodies.  The
details are standard, and may be found, for example, in Scharlemann's
survey article \cite{Scharlemann:survey}.  Complete details on many of
the results may be found, among other places, in
\cite{burton:thesis2012}.

\begin{define}\label{def:comprbody}
Let $S$ be a (possibly disconnected) closed, orientable surface with
genus at least $1$.  A \emph{compression body} is the result of
attaching a finite collection of 1--handles to $S\times [0,1]$ on the
boundary component $S \times \{ 1\}$, in a piecewise linear manner; we
require that our compression bodies be connected.

If $C$ is a compression body, the \emph{negative boundary} $\bdy_- C$
is $S\times \{0\}$.  The \emph{positive boundary} $\bdy_+ C$ is $\bdy
C \setminus \bdy_- C$.

Note that $\bdy_- C$ will consist of a disjoint union of surfaces, of
genus $m_1, \dots, m_k$, and $\bdy_+ C$ will be a genus $n$ surface
with $n \geq \sum m_i $.  We will call such a compression body an
\emph{$(m_1, \dots, m_k ;n)$--compression body}.  
\end{define}

In this paper, we will only consider examples with $S$ connected,
hence we are interested in $(m;n)$--compression bodies, with $n \geq
m$.  Usually we will set $m=1$.

Any two $(m;n)$--compression bodies are homeomorphic
\cite{burton:thesis2012}.  Hence, for fixed $m,n$, we will usually
refer to \emph{the} $(m;n)$--compression body.

\begin{define}\label{def:disksys}
A \emph{system of disks} for a compression body $C$ is a collection
of properly embedded essential disks $\{D_1, \dots, D_n\}$ such that
the complement of a regular neighborhood of $\bigcup_{i=1}^n D_i$ in
$C$ is homeomorphic to the disjoint union of a collection of balls and
the manifold $\bdy_- C \times [0,1]$.  A system of disks is
\emph{minimal} if the complement of a regular neighborhood of
$\bigcup_{i=1}^n D_i$ in $C$ is homeomorphic to $\bdy_- C \times
     [0,1]$.
\end{define}

Each $(m;n)$--compression body admits a minimal system of disks, and
such a system of disks will contain exactly $n-m$ disks
\cite{burton:thesis2012}.  In fact, provided $n-m \geq 2$, the
$(m;n)$--compression body will actually admit countably many systems
of disks, each related by some sequence of \emph{disk slides}, as in
the following definition.

\begin{define}\label{def:diskslide}
Let $C$ be an $(m;n)$-compression body with $n-m \geq 2$, and let $\DD
= \{D_1, \dots, D_{n-m}\}$ be a minimal system of disks for $C$.  Let
$N$ be a regular neighborhood of $\DD$.  Then $C \setminus N$ is
homeomorphic to $\partial_-C \times [0,1]$, with the (positive)
boundary component $\bdy_-C \times \{1\}$ containing pairs of disks,
denoted $E_i$ and $E_i'$, which are parallel to $D_i$.

Let $\alpha$ be an arc in $\partial_-C \times \{1\}$, with one
endpoint on one disk, say $E_i$, and the other endpoint on another
disk, say $E_j$, and otherwise disjoint from all the disks $E_k \cup
E_k'$.  Let $N'$ be a regular neighborhood in $C$ of $E_i \cup \alpha
\cup E_j$. Then $\overline{N}'$ is a closed ball which intersects
$\partial_+ C$ in a thrice--punctured sphere. The set $\partial N'
\backslash \partial C$ consists of three disks: one parallel to $D_i$,
one parallel to $D_j$, and another disk $D_i *_{\alpha} D_j$. Let
$\mathbf{D}' = \{D_1, \hdots, \widehat{D}_i, \hdots, D_n, D_i
*_{\alpha} D_j\}$, where as usual $\widehat{D}_i$ means remove $D_i$
from the collection.  Then $\DD'$ is also a minimal system of disks.
It is said to be a \emph{disk slide} of $\DD$.  See
Figure~\ref{fig:DiskSlide}.
\end{define}

\begin{figure}
\begin{center}
  \input{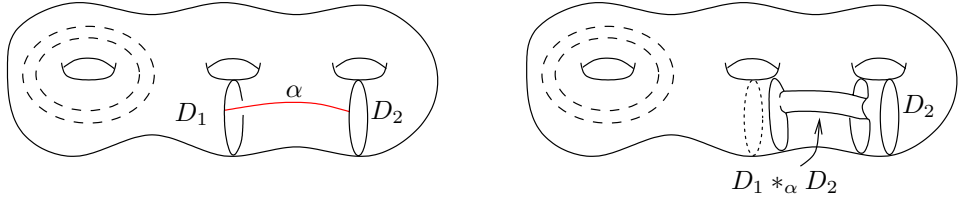}
\caption{A disk slide in a $(1;3)$--compression body.}
\label{fig:DiskSlide}
\end{center}
\end{figure}

Associated to a minimal system of disks for $C$ is a system of arcs, as
follows.  

\begin{define}\label{def:spine}
Let $K$ be a graph embedded in a compression body $C$ whose only
vertices are valence one, embedded in $\bdy_- C$.  If $C$ deformation
retracts to $\bdy_- C \cup K$, then we say $K$ is a \emph{spine} for
$C$.  A spine $K$ is \emph{dual} to a minimal system of disks $\DD$ if
each edge of $K$ intersects a single disk of $\DD$ exactly once, and
each disk in $\DD$ intersects an edge of $K$.
\end{define}

Given any minimal system of disks $\DD$ for a compression body $C$,
there is always a spine dual to $\DD$, and the spine is unique up to
isotopy \cite{burton:thesis2012}.

\begin{define}\label{def:core-tunnel}
Let $C$ be a compression body, and let $K$ be a spine dual to a
minimal system of disks $\DD$ for $C$.  The edges of $K$ are arcs
running from $\bdy_-C$ to $\bdy_-C$.  We call such an arc a \emph{core
  tunnel} for $C$, and we say the spine $K$ is a \emph{core tunnel
  system}, or simply a \emph{tunnel system}, for $C$.
\end{define}

Just as there are countably many minimal systems of disks $\DD$ for a
compression body $C$, there are also countably many tunnel systems.
However, we will work frequently with one particular system, given by
the following lemma.

\begin{lemma}\label{lemma:cores}
Recall that the $(m;n)$--compression body $C$ is obtained by attaching
$(n-m)$ 1--handles to the $S\times \{1\}$ component of $S\times I$,
where $S$ is a genus--$m$ surface.  For each $i=1, \dots, n-m$, let
$e_i$ be an edge at the core of the corresponding 1--handle, extended
vertically through $S\times [0,1]$ to have boundary on $S\times
\{0\}$.  Then $\bigcup_{i=1}^{n-m} e_i$ forms a tunnel system for $C$,
and each $e_i$ is a core tunnel.
\end{lemma}

\begin{proof}
We need to show that $K = \bigcup_{i=1}^{n-m} e_i$ is a spine for $C$
which is dual to a minimal system of disks.  Note that if we let $D_i$
be a cross--sectional disk in a 1--handle, then the collection $\{D_1,
\dots, D_{n-m}\}$ forms a minimal disk system dual to $K$.  Moreover,
the $i$-th 1--handle deformation retracts to the $e_i$ at its core.
Extending this to all of $C$, we see that $C$ deformation retracts to
$\bdy_-C \cup K$.  Hence each $e_i$ is a core tunnel for $C$, and the
collection of the $e_i$ forms a tunnel system.
\end{proof}

\begin{define}\label{def:StandardGen}
Inside the $(m;n)$--compression body $C$, for each core tunnel $e_i$
of Lemma \ref{lemma:cores}, connect the endpoints of $e_i$ by an arc
in $S\times \{0\}$.  The result is a loop $\gamma_i$ in $C$.  In fact,
if we let $\alpha_1, \beta_1, \dots, \alpha_{2m-2}, \beta_{2m-2}$ be
loops generating $\pi_1(S)$, then the loops $\alpha_j, \beta_j,
\gamma_i$, for $j=1, \dots, 2m-2$, and $i=1, \dots, n-m$, generate
$\pi_1(C)$ (after we extend the $\gamma_i$ to meet a common basepoint
on $S\times \{0\}$).  We call such a collection of generators
\emph{standard} generators for $\pi_1(C)$.
\end{define}

Hereafter, we will primarily work with the $(1;n+1)$--compression
body.  Standard generators will be written as $\alpha, \beta,
\gamma_1, \dots, \gamma_n$.

\subsection{Hyperbolic geometry of compression bodies}
We are interested in relating the topology of a compression body $C$
to hyperbolic geometry on $C$.  Specifically, we wish to understand
the behavior of geodesic arcs homotopic or isotopic to a core tunnel
in a hyperbolic structure on the interior of $C$.  We obtain a
hyperbolic structure by taking a discrete, faithful representation
$\rho \co \pi_1(C) \to \PSL(2, \CC)$, and considering the manifold
$\HH^3 / \rho(\pi_1(C))$.

Recall that a discrete subgroup $\Gamma < {\rm PSL}(2,\CC)$
is \emph{geometrically finite} if $\HH^3/\Gamma$ admits a
finite--sided, convex fundamental domain.  In this case, we will also
say that the manifold $\HH^3/\Gamma$ is \emph{geometrically finite}.

A discrete subgroup $\Gamma < \PSL(2,\CC)$ is \emph{minimally
  parabolic} if it has no rank one parabolic subgroups.  In other
words, a discrete, faithful representation $\rho\co \pi_1(M) \to
\PSL(2,\CC)$ will be minimally parabolic if and only if whenever
$\rho(g)$ is parabolic, $g$ is conjugate to an element of the
fundamental group of a torus boundary component.

\begin{define}
  A discrete, faithful representation $\rho\co\pi_1(M)\to \PSL(2,\CC)$
  is a \emph{minimally parabolic geometrically finite uniformization
    of $M$} if $\rho(\pi_1(M))$ is minimally parabolic and
  geometrically finite, and $\HH^3/\rho(\pi_1(M))$ is homeomorphic to
  the interior of $M$.
\label{def:gf}
\end{define}

We must describe the Ford domain of a minimally parabolic
geometrically finite uniformization $\rho$ of a $(1;n+1)$--compression
body.

\begin{define}
Let $A \in PSL(2,\CC)$ be loxodromic, and let $H$ be any horosphere
about infinity in upper half space $\HH^3$.  Then the \emph{isometric
sphere} corresponding to $A$, which we write $I(A)$, is the set of
points in $\HH^3$ equidistant from $H$ and $A^{-1}(H)$.
\label{def:isosphere}
\end{define}

If $A = \begin{bmatrix}a & b \\ c & d \end{bmatrix},$ then it is well
known that the isometric sphere $I(A)$ is the Euclidean hemisphere with
center $-d/c$ and radius $1/|c|$ (see, for example
\cite{LackenbyPurcell}).  

\begin{define}
Let $\Gamma$ be a discrete subgroup of $\PSL(2,\CC)$, with
$\Gamma_\infty < \Gamma$ the subgroup fixing the point at infinity in
$\HH^3$.  For $g \in \Gamma \setminus \Gamma_\infty$, let $B_g$ denote
the open half ball bounded by $I(g)$, and define the \emph{equivariant
  Ford domain} $\F$ to be the set
$$\F = \HH^3 \setminus ( \bigcup_{g \in \Gamma \setminus \Gamma_\infty}
B_g ).$$

A \emph{vertical fundamental domain} for a parabolic group
$\Gamma_\infty$ fixing the point at infinity in $\HH^3$ is a choice of
(connected) fundamental domain for the action of $\Gamma_\infty$ which
is cut out by finitely many vertical geodesic planes in $\HH^3$.

The \emph{Ford domain} of $\Gamma$ is defined to be the intersection
of $\F$ with a vertical fundamental domain for the action of
$\Gamma_\infty$.
  \label{def:ford}
\end{define}

The Ford domain is not canonical, because there is a choice of
vertical fundamental domain.  However, the region $\F$ is canonical.

Bowditch showed that if $\Gamma < \PSL(2,\CC)$ is geometrically
finite, then every convex fundamental domain for $\HH^3/\Gamma$ has
finitely many faces \cite[Proposition 5.7]{bowditch}.  In particular,
when $\Gamma$ is geometrically finite, there will only be finitely
many faces in a Ford domain.  This means that for all but finitely
many elements $g \in \Gamma \setminus \Gamma_\infty$ the isometric
sphere $I(g)$ is completely covered by some other isometric sphere.
We formalize this in a definition.

\begin{define}
  An isometric sphere $I(g)$ is said to be \emph{visible} if there
  exists an open set $U\subseteq \HH^3$ such that $U \cap I(g) \neq
  \emptyset$, and the hyperbolic distances satisfy
$$d(x, h\inv(H)) \geq d(x,H) = d(x,g\inv H)$$ for every $x \in U \cap
  I(g)$ and $h \in \Gamma \backslash \Gamma_\infty$, where $H$ is some
  horosphere about infinity.
  \label{def:visible}
\end{define}

A proof of the following fact may be found in \cite{LackenbyPurcell}.

\begin{lemma}
For $\Gamma$ discrete, the following are equivalent.
\begin{enumerate}
\item The isometric sphere $I(g)$ is visible.
\item There exists a two dimensional cell of the cell structure on
  $\mathcal{F}$ contained in $I(g)$.
\item $I(g)$ is not contained in $\bigcup_{h \in
  \Gamma\backslash(\Gamma_\infty \cup \Gamma_\infty g)} \bar{B}_h.$
\end{enumerate}
\end{lemma}

We will mainly be considering uniformizations of a $(1;
n+1)$--compression body where the Ford domain is of a particularly
simple type, which occurs in the following example.

\begin{example}
  Let $C$ be a $(1;3)$-compression body. Then $\pi_1(C) \cong ( \ZZ
  \times \ZZ) * \ZZ * \ZZ$.  We will choose generators $\alpha, \beta,
  \gamma$ and $\delta$ for $\pi_1(C)$, where $\alpha$ and $\beta$
  generate the $\ZZ \times \ZZ$ subgroup.  Consider the representation
\[
\begin{array}{ll}
\vspace{.1in}

\rho(\alpha) = \begin{bmatrix}1 & 100 \\ 0 & 1 \end{bmatrix} &
\rho(\beta) = \begin{bmatrix} 1 & 100i\\0 & 1 \end{bmatrix}\\

\rho(\gamma) = \begin{bmatrix} 0 & 1 \\ -1 & -5i \end{bmatrix} &
\rho(\delta) = \begin{bmatrix} -5-5i & -26-25i\\ 1 & 5 \end{bmatrix}
\end{array}
\]

Let $\Gamma_\infty = \langle \rho(\alpha), \rho(\beta) \rangle <
\PSL(2,\CC)$.  Here we have chosen $\rho(\alpha)$ and $\rho(\beta)$
somewhat arbitrarily so that they give a very large parabolic
translation length.  Drawing the isometric spheres corresponding to
$\rho(\gamma^{\pm1})$, and $\rho(\delta^{\pm1})$ gives us the picture
in Figure \ref{fig:3dFord}.

\begin{figure}
\centering
\includegraphics[scale=0.5]{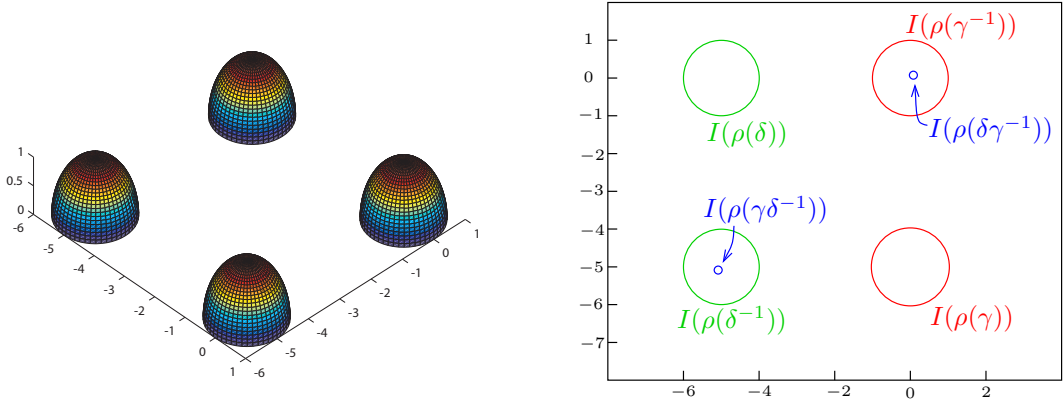}
\hspace{.2in}
\input{figures/initsimpleFord.pspdftex}
\caption{Isometric spheres from Example \ref{ex:SimpleFord} are shown,
  in 3--dimensions on the left, and their 2--dimensional intersections
  with $\CC$ on the right.}
\label{fig:3dFord}
\end{figure}

We will see that other isometric spheres, besides the translates under
$\Gamma_\infty$ of $I(\rho(\gamma^{\pm 1}))$ and $I(\rho(\delta^{\pm
  1}))$, will be invisible, hidden underneath these isometric spheres.
For example, the isometric spheres $I(\rho(\gamma\delta\inv))$ and
$I(\rho(\delta\gamma\inv))$ shown in Figure \ref{fig:3dFord} are
invisible.  Hence $\rho$ will give a minimally parabolic geometrically
finite uniformization of $C$ whose Ford domain is as in Figure
\ref{fig:3dFord}.
\label{ex:SimpleFord}
\end{example}

We now generalize Example \ref{ex:SimpleFord}.  To do so, set up the
following notation.

Let $C_n$ denote the $(1;n+1)$--compression body.  Hence $\pi_1(C_n)
\cong (\ZZ \times \ZZ) * F_n$, where $F_n$ denotes the free group on
$n$ letters $\ZZ * \dots * \ZZ$.  Let $\alpha, \beta, \gamma_1, \dots,
\gamma_n$ be generators of $\pi_1(C_n)$, with $\alpha$ and $\beta$
generating the $(\ZZ \times \ZZ)$ subgroup, and $\gamma_1, \dots,
\gamma_n$ standard, as in Definition \ref{def:StandardGen}, coming
from a tunnel system of $C_n$, as in Lemma \ref{lemma:cores}.
Finally, let $\rho \co \pi_1(C_n) \to \PSL(2,\CC)$ be a discrete
representation, taking $\alpha$ and $\beta$ to parabolics fixing
infinity, generating the subgroup $\Gamma_\infty = \langle
\rho(\alpha), \rho(\beta)\rangle < \PSL(2,\CC)$.

\begin{define}
With notation as above, suppose the isometric spheres corresponding to
$\rho(\gamma_1^{\pm 1}), \dots, \rho(\gamma_n^{\pm 1})$ and their
translates under $\Gamma_\infty$ are all pairwise disjoint, with none
properly contained in a half--ball bounded by one of the others.  Then
we say that $\rho$ gives a \emph{simple Ford domain} for $C_n$.
  \label{def:SimpleFord}
\end{define}

Note that Example \ref{ex:SimpleFord} is an example of a simple Ford
domain.  The use of the words ``Ford domain'' in Definition
\ref{def:SimpleFord} is justified by the following lemma.

\begin{lemma}\label{lemma:SimpleFord}
Suppose $\rho \co \pi_1(C_n) \to \PSL(2,\CC)$ gives a simple Ford
domain for $C_n$.  Then $\rho$ gives a minimally parabolic,
geometrically finite uniformization of $C_n$.  Moreover, after
possibly replacing the $\gamma_i$ by multiples of $\gamma_i$ with
elements in the $(\ZZ \times \ZZ)$ subgroup of $\pi_1(C_n)$, the
isometric spheres corresponding to $\rho(\gamma_1^{\pm 1}), \dots,
\rho(\gamma_n^{\pm 1})$, along with a choice of vertical fundamental
domain for $\Gamma_\infty$, cut out a Ford domain.
\end{lemma}

\begin{proof}
Choose a vertical fundamental domain for $\Gamma_\infty$.  Recall that
the center of the isometric sphere $I(\rho(\gamma_i))$ lies at the
point $\rho(\gamma_i\inv)(\infty)$.  We may multiply each $\gamma_i$,
$i=1, \dots, n$ on the right by some $w_i \in (\ZZ\times\ZZ)$ so that
the center $\rho(w_i\inv \gamma_i\inv)(\infty)$ of the isometric
sphere $I(\rho(\gamma_i w_i))$ lies inside the chosen vertical
fundamental domain.  Note that $w_i$ is a word in $\alpha$ and
$\beta$, and so $\alpha, \beta, \gamma_1 w_1, \dots, \gamma_n w_n$
still generate $\pi_1(C_n)$, and the $\gamma_i w_i$ still give
isometric spheres whose translates under $\Gamma_\infty$ are pairwise
disjoint, as in Definition \ref{def:SimpleFord}.  Thus without loss of
generality, we may assume that the centers of the $I(\rho(\gamma_i))$
are all contained in our chosen vertical fundamental domain.

Next, consider the isometric spheres corresponding to
$\rho(\gamma_i\inv)$, for $i=1, \dots, n$.  We may multiply each
$\gamma_i$, $i=1, \dots, n$, on the left by some $x_i \in (\ZZ\times
\ZZ)$ so that the center $\rho(x_i\inv \gamma_i)(\infty)$ of the
isometric sphere $I(\rho(\gamma_i\inv x_i))$ lies inside the chosen
vertical fundamental domain.  Note also that the center
$\rho(\gamma_i\inv x_i)(\infty)$ of $I(\rho(x_i\inv \gamma_i))$ is the
same as the center $\rho(\gamma_i\inv)(\infty)$ of $I(\rho(\gamma_i))$,
because $\rho(x_i)$ fixes $\infty$, so when we replace each $\gamma_i$
by $x_i\inv \gamma_i$ we obtain generators of $\pi_1(C_n)$ such that
the corresponding isometric spheres $I(\rho(\gamma_i^{\pm 1}))$ all
have centers within our chosen vertical fundamental domain.  Moreover,
note that these isometric spheres still satisfy the definition of a
simple Ford domain.  

Now, let $P$ be the intersection of the chosen vertical fundamental
domain with the exterior of the isometric spheres corresponding to
$\rho(\gamma_i^{\pm 1})$.  Because none of these isometric spheres is
contained inside another, and because they do not intersect, $P$ is
homeomorphic to a 3--ball, marked with simply connected \emph{faces},
which faces correspond to the faces of the vertical fundamental domain
and to each isometric sphere $I({\rho(\gamma^{\pm 1})})$.

Identify vertical sides of $P$ by elements of $\Gamma_\infty$, and
glue $I({\rho(\gamma_i)})$ to $I(\rho(\gamma_i\inv))$ via
$\rho(\gamma_i\inv)$, for each $i=1, \dots, n$.  This glues faces of
$P$ by isometry, and since the intersections of faces (\emph{edges} of
$P$) are only on the vertical fundamental domain, the Poincar{\'e}
Polyhedron Theorem (cf \cite[Theorem 2.21]{LackenbyPurcell},
\cite{EpsteinPetronio}), implies that the result of applying these
gluings to $P$ is a smooth manifold $M$, with $\pi_1(M) \cong
\pi_1(C_n)$ generated by face pairings.  Moreover, by \cite[Theorem
  2.22]{LackenbyPurcell}, $P$ must be a Ford domain for $M \cong
\HH^3/\Gamma$, and by \cite[Lemma 2.18]{LackenbyPurcell}, it is
minimally parabolic.

Hence, to show that this gives a minimally parabolic geometrically
finite uniformization of $C_n$, it remains only to show that $M$ is
homeomorphic to $C_n$.  We show this by considering gluing faces of
$P$ one at a time.

First, glue faces corresponding to the vertical fundamental domain.
Since $\Gamma_\infty$ is a rank--2 parabolic group, the result is
homeomorphic to $T^2 \times \RR$, where $T^2$ is the torus.  Now,
notice that when we glue the face $I({\rho(\gamma_i)})$ to
$I({\rho(\gamma_i\inv)})$, the result is topologically equivalent to
attaching a 1--handle.  Hence, when performing the gluing one by one
for each $i=1, \dots, n$, we obtain a manifold homeomorphic to $C_n$.
\end{proof}

\subsection{Tunnel systems and hyperbolic geometry}

We are interested in studying a tunnel system for a manifold, and we
need to identify a tunnel system in a geometrically finite minimally
parabolic uniformization.

For $\gamma \in \PSL(2,\CC)$ that does not fix the point at infinity
in $\HH^3$, there is a geodesic $e_\gamma$ running from
$\gamma^{-1}(\infty)$ to $\infty$ in $\HH^3$.  This geodesic
$e_\gamma$ will meet the center of the Euclidean hemisphere
$I(\gamma)$.  We say that $e_\gamma$ is the \emph{geometric dual} of
the isometric sphere $I(\gamma)$.  We also refer to $e_\gamma$ as the
\emph{geodesic dual} to the isometric sphere $I(\gamma)$.

\begin{lemma}\label{lemma:DualsAreTunnels}
Let $C$ denote the $(1;n+1)$--compression body, where $\pi_1(C) \cong
(\ZZ \times \ZZ) * F_n$ has generators $\alpha, \beta, \gamma_1,
\cdots, \gamma_n$, with $\alpha$ and $\beta$ generating the $(\ZZ
\times \ZZ)$ subgroup, and $\gamma_1, \dots, \gamma_n$ standard,
coming from a tunnel system, as in Definition \ref{def:StandardGen}.
Let $\rho\co \pi_1(C) \to \PSL(2,\CC)$ be a minimally parabolic,
geometrically finite uniformization of $C$, normalized such that
$\rho(\alpha)$ and $\rho(\beta)$ fix the point of infinity of $\HH^3$.
Finally, let $\tilde{d}_i$ be the geodesic dual to the isometric
sphere $I(\rho(\gamma_i\inv))$.  Then under the quotient
action of $\Gamma$, the images of the dual edges $\tilde{d}_i$ are
homotopic to a spine of $C$.  Hence these geometric edges are
homotopic to a tunnel system.
\end{lemma}

\begin{proof}
We will show that the images of the geodesics $\tilde{d}_i$ are
homotopic to the core tunnels $e_i$ of Lemma \ref{lemma:cores}, and
this will be enough to prove the lemma.

In the topological manifold $C$, take the closure a regular
neighborhood $N$ of $\partial_-C$ so that the closure $\overline{N}$
is homeomorphic to $\partial_-C \times [0,1]$. Choose $p = (p',1) \in
\partial_-C \times\{1\}$ and let $q = (p',0) \in \partial_-C\times
\{0\}$. Let $f\co [0,1] \to C$ be the straight line from $p$ to $q$.

In $\HH^3$, choose a vertical fundamental domain $D$ for $\Gamma_\infty
= \langle \rho(\alpha), \rho(\beta)\rangle$.  As in the proof of Lemma
\ref{lemma:SimpleFord}, we may replace the $\gamma_i$ by products
$w_i\cdot\gamma_i\cdot v_i$ where $w_i,v_i$ are in $\Gamma_\infty$,
and thereby assume that $D$ contains $\rho(\gamma_i^{\pm 1})(\infty)$
for all $i=1, 2, \cdots, n$ (or rather, these points are contained in
the closure of $D$ in $\HH^3 \cup \CC \cup \{\infty\}$).  Note that
the replacement doesn't affect the argument, since under these
translations, dual geodesics $\tilde{d}_i$ still map to the same
geodesic in $\HH^3/\rho(\pi_1(C))$.

The lift $\tilde{p}$ of $p$ into $D$ is a point on a horoball $H$
about $\infty$.  For each loxodromic $\rho(\gamma_i)$, define
$\tilde{p}_i = \rho(\gamma_i)(\tilde{p})$.  The point $\tilde{p}_i$
lies on a horosphere centered at $\rho(\gamma_i)(\infty)$. For each $i
= 1, 2, \hdots, n$, let $\tilde{g}_i$ be a geodesic arc in $D$ from
$\tilde{p}$ to $\tilde{p}_i$.  Under the action of $\Gamma$, the arc
$\tilde{g}_i$ becomes a loop in the homotopy class of $\gamma_i$ in
$C$.

Let $\tilde{f}_i$ be a geodesic arc in $D$ from $\tilde{p}_i$ to
$\gamma_i(\infty)$, and let $\tilde{f}'_i$ be a geodesic arc from
$\infty$ to $\tilde{p}$. Under the action of $\Gamma$, the closure of
the quotient of the arcs $\tilde{f}_i$ and $\tilde{f}_i'$ in $C$
become arcs from $p$ to points on $\partial_-C$, which are homotopic
to $f$ rel $p$, and the homotopy may be taken to keep an endpoint of
each of the arcs on $\partial_-C$.

Set $\tilde{h}_i$ to be the arc $\tilde{f}'_i$ followed by
$\tilde{g}_i$ followed by $\tilde{f}_i$.  Then $\tilde{h}_i$ runs from
$\infty$ to $\gamma_i(\infty)$.  Therefore $\tilde{h}_i$ is homotopic
to $\tilde{d}_i$.

On the other hand, under the action of $\Gamma$, $\tilde{h}_i$ is
mapped to a loop with endpoints on $\partial_-C$ in the homotopy class
of $\gamma_i$.  Allowing the endpoints of this loop to move on
$\bdy_-C$, we may homotope to the arc $e_i$, which is a core tunnel
from Lemma \ref{lemma:cores}, corresponding to the standard generator
$\gamma_i$ coming from Definition \ref{def:StandardGen}.
\end{proof}

Lemma \ref{lemma:DualsAreTunnels} shows only that the geodesic duals
to isometric spheres corresponding to a set of generators gives a set
\emph{homotopic} to a tunnel system.  We are interested in examples of
when these geodesics are \emph{isotopic} to a tunnel system.  One
example of when this will occur comes from the following lemma.

\begin{lemma}\label{lemma:SimpleTunnelSystem}
With notation as in Lemma \ref{lemma:DualsAreTunnels}, if the Ford
domain is simple and the isometric sphere corresponding to each
$\rho(\gamma_i)$ is visible, then in the quotient manifold $C \cong
\HH^3 / \rho(\pi_1(C))$ the images of the $\tilde{d}_i$ are isotopic
to a spine of $C$.  Hence these edges form a geodesic tunnel system.
\end{lemma}

\begin{proof}
Let $\F$ be the equivariant Ford domain.  Let $H \subset \F$ be an
embedded horoball about infinity.  As in \cite[Lemma
  3.11]{LackenbyPurcell}, we construct an equivariant deformation
retract of $\F \setminus H$ onto the union of the geodesic arcs
$\Gamma_\infty (\tilde{d}_i \cap (\F \setminus H))$ and $\partial H$.
We do so in two steps.

First, by Lemma \ref{lemma:SimpleFord}, we may assume that the
isometric spheres corresponding to $\rho(\gamma_i^{\pm 1})$, $i=1,
\dots, n$, along with a vertical fundamental domain for
$\Gamma_\infty$ cut out a Ford domain.  The boundaries of the
isometric spheres give embedded circles on $\CC$, which bound disjoint
disks $D_1, D'_1, D_2, D'_2, \dots, D_{n}, D'_n$ on $\CC$, with $D_i$
corresponding to $\rho(\gamma_i)$ and $D'_i$ corresponding to
$\rho(\gamma_i^{-1})$.  Now, choose a value of $\epsilon>0$ such that
for the disks $E_i$ and $E_i'$, which are the $\epsilon$ neighborhoods
of $D_i$ and $D_i'$, respectively, the collection $E_1, E_1', E_2,
E_2', \dots, E_n, E_n'$ on $\CC$ still consists of disjoint disks.
Take the vertical projection of these $E_i, E_i'$ onto the boundary of
the horoball $H$; we will continue to denote these disks on $\bdy H$
by $E_1, E_1', \dots, E_n, E_n'$.  For each $i = 1, 2, \dots, n$,
consider the frustrum $C_i$ of the solid (Euclidean) cone in $\HH^3$
which intersects $\bdy H$ in the disk $E_i$, and intersects $\CC$ in
the disk $D_i$.  Similarly, we have the frustrum $C_i'$ meeting $\bdy
H$ in $E_i'$ and meeting $\CC$ in $D_i'$.  By choice of $\epsilon$,
the sets $C_1 \cap (\F \setminus H), C_1' \cap (\F \setminus H),
\dots, C_n \cap (\F \setminus H), C_n' \cap (\F \setminus H)$, as well
as their translates under $\Gamma_\infty$, are disjoint in $\F
\setminus H$.  Let
$$C = \Gamma_\infty \left( \bigcup_{i=1}^n ( C_i \cup C_i') \cap (\F
  \setminus H) \right) \subset (\F \setminus H).$$

The first step of the homotopy is to map $\F \setminus (C \cup H)$
onto $\partial (C \cup H)$ via the vertical line homotopy.  That is,
each point $x$ in $\F \setminus (C \cup H)$ lies on a vertical line
through $\infty$, and this line will meet $\partial (C\cup H)$ exactly
once.  Let $L_t(x)$ be the point on this vertical line, so that
$L_0(x)$ is the identity and $L_1(x)$ lies on $\partial (C \cup H)$.
Note the map $L_t$ is continuous, equivariant under the action of
$\rho(\pi_1(C))$, and descends to a continuous map in the quotient
$\HH^3 / \rho(\pi_1(C))$.

The second step is to deformation retract $C \cap (\F\setminus H)$
onto the set
$$\left(\partial H \cup \left(\bigcup_{i=1}^n \Gamma_\infty (\tilde{d}_i \cup
\rho(\gamma_i)(\tilde{d}_i))\right)\right) \cap (\F \setminus H).
$$ Since $C_i$ and $C'_i$ form regular neighborhods of $\tilde{d}_i$
and $\rho(\gamma_i)(\tilde{d}_i)$, respectively, there is a
deformation retract sending each $C_i \cap (\F\setminus H)$ and $C_i'
\cap (\F\setminus H)$, $i=1, \dots, n$, onto the geodesic at its core.
Note by choice of $\epsilon$, we may perform these deformation
retracts simultaneously and equivariantly, since none of these cones
intersect in $\F\setminus H$.  It is clear that we can modify this
deformation retract to a deformation retract onto $\tilde{d}_i \cup
(\bdy H \cap C_i)$ or $\rho(\gamma_i)\tilde{d}_i \cup (\bdy H \cap
C_i')$, for $i=1, \dots, n$.  We let $f_t$ be the deformation retract
of the second step.  Then the deformation retract $L_t$ followed by
$f_t$ is the desired equivariant deformation retract.
\end{proof}

\section{Tunnel systems in compression bodies}

In this section we show that the geodesic duals in the Ford domain may
be made to intersect while retaining a geometrically finite
structure.

\begin{lemma}\label{lem:DualToInvisibleSphere}
Let $\gamma$ and $\delta$ be loxodromic generators of a
$(1;n)$--compression body $C$.  Suppose that under some geometrically
finite uniformization $\rho\co \pi_1(C) \to \PSL(2,\CC)$ of $C$, the
faces of the Ford domain corresponding to $\rho(\delta^{\pm 1})$ and
$\rho((\delta\gamma\inv)^{\pm 1})$ are visible, and that the isometric
sphere $I({\rho(\gamma)})$ is contained in the Euclidean half--ball
bounded by the isometric sphere $I({\rho(\delta)})$.  Then the
geometric dual $\widetilde{g}$ to $I({\rho(\gamma)})$ in $\HH^3$ is
mapped to a geodesic $g$ under the quotient $\HH^3 \to \HH^3 /
\rho(\pi_1(C))$ with the property that $g$ lifts to geodesics in
$\HH^3$ containing the arcs:
\begin{enumerate}
\item $\alpha_1$, running from $\infty$ to a point on
  $I({\rho(\delta)})$ (a subarc of the geodesic dual to
  $I(\rho(\gamma))$), 
\item $\alpha_2$, running from a point on $I({\rho(\delta\inv)})$ to a
  point on $I({\rho(\gamma \delta\inv)})$ (a subarc of the geodesic
  from the center of $I(\rho(\delta\inv))$ to the center of
  $I(\rho(\gamma\delta\inv))$), 
\item and $\alpha_3$, running from $\infty$ to a point on
  $I({\rho(\delta\gamma\inv)})$ (a subarc of the geodesic dual to
  $I(\rho(\gamma\inv))$).  
\end{enumerate}
\end{lemma}

Lemma \ref{lem:DualToInvisibleSphere} is illustrated in
Figure~\ref{fig:DualToInvisible}.

\begin{figure}
\begin{center}
  \input{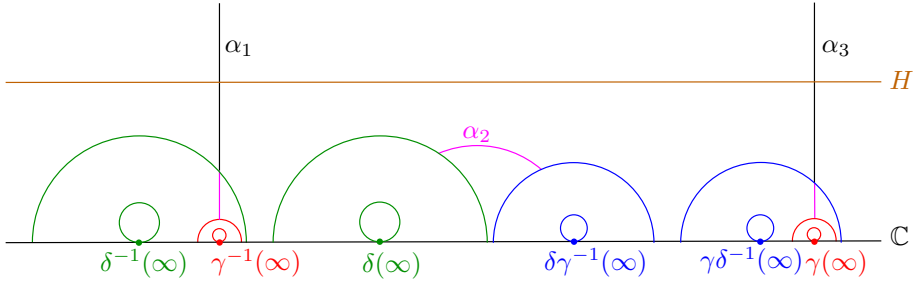}
\caption{Lift of $g$ consists of the arcs $\alpha_1$, $\alpha_2$, and
  $\alpha_3$.}
\label{fig:DualToInvisible}
\end{center}
\end{figure}

\begin{proof}
Since the uniformization $\rho\co \pi_1(C) \to \PSL(2,\CC)$ is applied
to each group element in the proof, we will suppress it for ease of
notation, writing $\gamma$, for example, rather than $\rho(\gamma)$.

Choose a horosphere $H$ about $\infty$.  Let $S$ be the set of points
in $\HH^3$ equidistant from $\delta\inv(H)$ and $\gamma\inv(H)$.  Let
$p_1$ be the intersection of $I(\delta)$ and $\widetilde{g}$, and let
$p_2$ be the intersection of $S$ and $\widetilde{g}$. Note that $p_2$
is contained inside the Euclidean half--ball bounded by $I(\delta)$,
since $I(\gamma)$ is contained inside that half--ball.

Now apply $\delta$ to this picture.  Under $\delta$, the horosphere
$\delta\inv(H)$ is mapped to $H$, and $H$ is mapped to $\delta(H)$,
and so the isometric sphere is $I(\delta)$ mapped to $I({\delta\inv})$
isometrically.  Likewise, $S$ gets mapped isometrically to
$I({\gamma\delta\inv})$.  The geodesic dual $\widetilde{g}$ is mapped
to the geodesic running from $\delta(\infty)$ to $\delta
\gamma\inv(\infty)$.  These are exactly the centers of the isometric
spheres $I(\delta\inv)$ and $I(\gamma\delta\inv)$, respectively.  Now
$\delta(\widetilde{g})$ is a geodesic which passes through
$\delta(p_1) \in I({\delta\inv})$ and $\delta(p_2) \in
I({\gamma\delta\inv})$.

In a similar manner as above, apply $\gamma$. The isometric sphere
$I({\gamma})$ is mapped to $I({\gamma\inv})$, and $S$ is mapped to
$I({\delta\gamma\inv})$.  The geodesic dual $\widetilde{g}$ gets mapped
to the geodesic dual to $I({\gamma\inv})$. Therefore $\widetilde{g}$
gets mapped to an arc containing the vertical line from a point on
$I({\gamma\inv})$ to $\infty$.

Now $\gamma(\widetilde{g})$, $\delta(\widetilde{g})$, and
$\widetilde{g}$ are mapped to the same geodesic $g$ in the quotient
$\HH^3/\rho(\pi_1(C))$.  The arcs $\alpha_1$, $\alpha_2$, and
$\alpha_3$ are just the portions of these geodesics which lie above
$I({\delta^{\pm 1}})$ and $I({(\delta\gamma\inv)^{\pm 1}})$.
\end{proof}

\begin{prop}\label{prop:OneIntersectingGeodesic}
There exists a geometrically finite, minimally parabolic
uniformization $\rho$ of a $(1;3)$--compression body $C$, and a
generator $\xi$ of the free part of $\pi_1(C)$ such that the image
of the geometric dual to $I({\rho(\xi)})$ under the action of
$\rho(\pi_1(C))$ has a self--intersection.
\end{prop}

\begin{proof}
We prove this by giving a specific example.
Recall that $\pi_1(C) \cong (\ZZ \times \ZZ) * \ZZ * \ZZ$.  We will
let $\alpha$, $\beta$, $\gamma_1$, and $\gamma_2$ generate $\pi_1(C)$,
such that $\alpha$ and $\beta$ generate the $\ZZ \times \ZZ$ subgroup.

We will consider a family of representations $\rho_t\co \pi_1(C) \to
\PSL(2,\CC)$ for which $\rho_t(\alpha)$, $\rho_t(\beta)$, and
$\rho_t(\gamma_1)$ are constant, and $\rho_t(\gamma_2)$ varies.  For
this example,
\[ \begin{array}{lcllcl}
\vspace{2mm} \rho_t(\alpha) & = & \begin{bmatrix}1 & 20 \\ 0 &
  1 \end{bmatrix},
& \rho_t(\beta) & = & \begin{bmatrix}20i & 1\\ 0 & 1 \end{bmatrix}
\end{array}
\]
These values of $\rho_t(\alpha)$ and $\rho_t(\beta)$ are chosen so
that the translation distances are large, basically so that we can
ignore the effect of these two elements on the changing Ford domain.

We let $B$ be the translation matrix $B = \begin{bmatrix} 1& 10\\ 0&
  1\end{bmatrix}$.  

We obtain $\rho_t(\gamma_1)$ by conjugation by $B$, and we set
$\rho_t(\gamma_2)$ to vary with $t$:
\[
\rho(\gamma_1) = B \begin{bmatrix} 0 & 1\\ -1 & 5-2i \end{bmatrix}
B^{-1}, \quad
\rho_t(\gamma_2) = \begin{bmatrix} 0 & 1\\ -1 & 5 +(t-2)i \end{bmatrix}.
\]

Note that the isometric spheres corresponding to $\rho_t(\gamma_1)$,
$\rho_t(\gamma_1\inv)$, $\rho_t(\gamma_2)$, and $\rho_t(\gamma_2\inv)$
all have radius $1$, and centers at $10$, $15-2i$, $0$, and $5+
(t-2)i$, respectively.  Hence for all $t \in [0,4]$, none of these
isometric spheres intersect.  Similarly, since the translation
distances of $\rho_t(\alpha)$ and $\rho_t(\beta)$ are large, no
translates of these isometric spheres under $\Gamma_\infty = \langle
\rho_t(\alpha), \rho_t(\beta)\rangle$ will intersect.  So a vertical
fundamental domain as well as isometric spheres corresponding to
$\rho_t(\gamma_1^{\pm 1})$ and $\rho_t(\gamma_2^{\pm 1})$ cut out a
simple Ford domain, and by Lemma \ref{lemma:SimpleFord}, this is a
minimally parabolic geometrically finite uniformization for $C$.

Now set $\delta_1 = \gamma_1$, and $\delta_2 = \gamma_2\inv \gamma_1$.
Then $\alpha, \beta, \delta_1, \delta_2$ generate $\pi_1(C)$.
Moreover, the isometric sphere corresponding to $\rho_t(\delta_2)$ will
be contained in the Euclidean half--ball bounded by the isometric
sphere corresponding to $\rho_t(\delta_1)$.  See Figure
\ref{fig:CrossingGeodesic}.  

\begin{figure}
\begin{center}
  \input{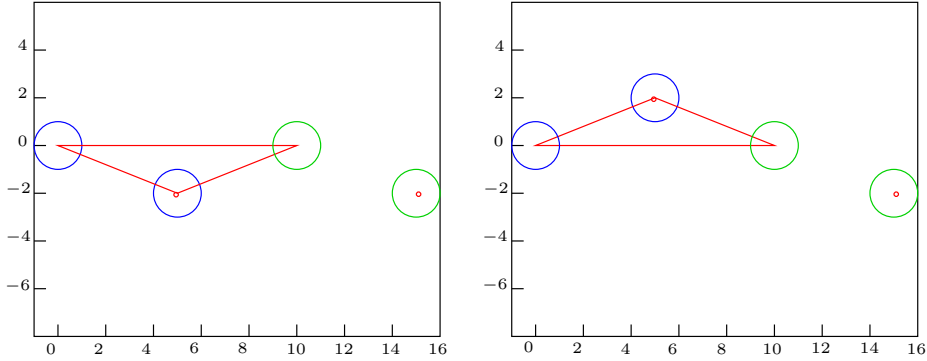}
\caption{When $t = 0$, the Ford domain is as pictured on the
  left. When $t = 4$, the Ford domain is as pictured on the right.}
\label{fig:CrossingGeodesic}
\end{center}
\end{figure}

Hence we have exactly the setup of Lemma
\ref{lem:DualToInvisibleSphere}, with $\delta_1$ playing the role of
$\delta$, and $\delta_2$ playing the role of $\gamma$.  Thus under the
action of $\rho_t(\pi_1(C))$, a portion of the geodesic dual to the
isometric sphere of $\rho(\delta_2)$ is mapped to a geodesic running
from a point $p_1(t)$ on the isometric sphere of $\rho_t(\delta_1\inv)
= \rho(\gamma_1\inv)$ to a point $p_2(t)$ on the isometric sphere of
$\rho_t(\delta_2 \delta_1\inv) = \rho_t(\gamma_2\inv)$.  

Define $p_3(t)$ to be the intersection of the geodesic dual to the
isometric sphere of $\rho_t(\delta_2\inv)$ with the isometric sphere
of $\rho_t(\gamma_2)$.
For each $t$ define a Euclidean triangle $T_t$ whose vertices are the
projections of $p_1(t)$, $p_2(t)$, and $p_3(t)$ onto $\CC$.  

Let the function $A\co [0,4] \to \RR$ give the signed area of $T_t$.
Carefully,
$$A(t) = \frac{1}{2} ( (p_1(t) - p_3(t) ) \times (p_2(t) - p_3(t))),$$
where $p_i(t) - p_j(t)$ is understood to be the 3--dimensional
Euclidean vector whose first two coordinates come from the
2--dimensional subtraction, and whose 3rd coordinate is zero, and
$\times$ denotes the usual Euclidean cross product on $\RR^3$.
Because the points $p_1(t), p_2(t)$ and $p_3(t)$ vary continuously
with $t$, $A(t)$ is a continuous function.

As can be seen in Figure~\ref{fig:CrossingGeodesic}, when $t = 0$ we
have $A(0) > 0$ since $p_3(0)$ must be below the line segment from
$p_1(0)$ to $p_2(0)$.  When $t = 4$ we obtain $A(4) < 0$, since
$p_3(4)$ must be above the line segment.  The Intermediate Value
Theorem guarantees that there is some $t_0 \in [0,4]$ for which
$A(t_0) = 0$, i.e.\ $p_1(t_0)$, $p_2(t_0)$ and $p_3(t_0)$ are colinear.
Hence when $t = t_0$, two of the geodesic arcs guaranteed by Lemma
\ref{lem:DualToInvisibleSphere} will intersect.  Thus the image of the
geometric dual to the isometric sphere of $\rho_{t_0}(\delta_2)$, 
under the action of $\rho_{t_0}(\pi_1(C))$, will have a
self--intersection.  
\end{proof}

We may generalize Proposition \ref{prop:OneIntersectingGeodesic} to
$(1; n+1)$--compression bodies.  The following result is Theorem
\ref{thm:main-compress} from the introduction, restated.

\begin{theorem}\label{thm:IntersectingGeodesics}
There exists a geometrically finite, minimally parabolic
uniformization $\rho$ of a $(1;n+1)$--compression body $C_n$ and a
choice of free generators $\delta_1,\hdots, \delta_n$ of $\pi_1(C_n)$
such that the geodesics $\tau_1, \dots, \tau_{n-1}$ obtained from the
geometric duals to isometric spheres corresponding to $\rho(\delta_1),
\dots, \rho(\delta_{n-1})$, respectively, each self--intersect.
\end{theorem}

\begin{proof}
As above, let $\alpha, \beta, \gamma_1, \dots, \gamma_n$ generate
$\pi_1(C_n)$, with $\alpha$ and $\beta$ generating the $\ZZ \times
\ZZ$ subgroup.  Set $A = \begin{bmatrix} 1 & 10\\ 0 &
  1 \end{bmatrix}.$

Let $t = (t_1, \dots, t_n) \in [0,4]\times \dots \times [0,4] \times
\{2\}$, and consider the $n$--parameter family of representations:
\[
\rho_t(\gamma_k) = A^{k-1}
\begin{bmatrix} 0 & 1 \\ -1 & 5+(t_k-2)i \end{bmatrix} A^{-(k-1)},
\quad \mbox{for } 1 \leq k \leq n,
\]
\[
\rho_t(\alpha) = \begin{bmatrix} 1 & 11n \\ 0 & 1 \end{bmatrix}, \quad
\rho_t(\beta) = \begin{bmatrix} 1 & 10i\\ 0 & 1 \end{bmatrix}.
\]

Note the isometric sphere corresponding to $\rho_t(\gamma_k\inv)$ has
radius $1$, and center $10(k-1)$.  The isometric sphere corresponding
to $\rho_t(\gamma_k)$ also has radius $1$, and center
$$5+ 10(k-1) + (t_k-2)i.$$ Hence for $t \in [0,4]\times \dots \times
[0,4] \times \{2\}$, all these isometric spheres are disjoint.
Moreover, $\rho_t(\alpha)$ and $\rho_t(\beta)$ are chosen to be large
enough so that the parabolic translates of these isometric spheres do
not intersect.  Thus a vertical fundamental domain as well as the
isometric spheres corresponding to $\rho_g(\gamma_k^{\pm 1})$ cut out
a simple Ford domain, and this is a minimally parabolic geometrically
finite uniformization for $C_n$.

Set $\delta_k = \gamma_k\inv \gamma_n$ for $1 \leq k <n$ and $\delta_n
= \gamma_n$.  The elements $\delta_1$, $\hdots$, $\delta_n$, $\alpha$,
$\beta$ still generate $\pi_1(C)$, but for $1 \leq k < n$, the
isometric sphere corresponding to $\delta_k\inv$ is not visible: it is
contained in the Euclidean half--ball bounded by the isometric sphere
corresponding to $\gamma_k$.  Similarly, for $1 \leq k < n$, the
isometric sphere corresponding to $\delta_k$ is contained in the
Euclidean half--ball bounded by the isometric sphere corresponding to
$\gamma_n$.  Thus Lemma \ref{lem:DualToInvisibleSphere} implies that
the geodesic running from the center of $I(\rho_t(\gamma_k\inv))$ to
the center of $I(\rho_t(\gamma_n\inv))$ maps to the image of the
geodesic dual to $I(\rho_t(\delta_k))$.

Now apply a similar argument to that in the previous proof.  The
Intermediate Value Theorem implies that for each $k$, $1 \leq k < n$,
there must be a $t_k \in [0,4]$ such that the geodesic dual to
$I({\rho_t(\delta_k)})$ has self--intersecting image.  Because for
$i\neq k$, varying $t_k$ has no effect on isometric spheres
corresponding to $\delta_k^{\pm 1}$, we may perform the above
procedure for each $k$ one at a time, to obtain the desired
uniformization.  
\end{proof}


\section{Finite volume tunnel number--$n$ manifolds}\label{sec:ftevol}
In this section, we use the results of the previous section to give
evidence that there exist tunnel number--$n$ manifolds with finite
volume and tunnel systems that come arbitrarily close to
self--intersecting.  

The rough idea of the proof is to take the compression body of Theorem
\ref{thm:IntersectingGeodesics}, and attach a handlebody to it in such
a way that the geometry of the compression body after attaching is
``close'' to the geometry before attaching.  This is accomplished in a
manner similar to that of Cooper, Lackenby, and Purcell in
\cite{CLP:LengthPaper}.

\subsection{Maximally cusped structures}

We recall definitions and results on maximally cusped geometrically
finite structures, because we will use these structures to build
manifolds with nearly self--intersecting tunnels.

\begin{definition}
A \emph{maximally cusped structure for $C$} is a geometrically finite
uniformization $\rho \co \pi_1(C) \to \PSL(2,\CC)$ of $C$ such that
every component of the boundary of the convex core of $\HH^3 /
\rho(\pi_1(C))$ is a 3--punctured sphere.
\end{definition}

In a maximally cusped structure for $C$, a full pants decomposition of
$\partial_+C$ is pinched to parabolic elements.  A theorem of Canary,
Culler, Hersonsky, and Shalen \cite{cchs}, extending work of McMullen
\cite{McMullen}, shows that the conjugacy classes of maximally cusped
structures for $C$ are dense on the boundary of all geometrically
finite structures on $C$.  To make this statement more precise, we
review the following definitions.

\begin{definition}
The \textit{representation variety} $V(C)$ of a compression body $C$
is the space of conjugacy classes of representations $\rho\co \pi_1(C)
\to \PSL(2,\CC)$, where $\rho$ sends elements of $\pi_1(\partial_-C)$
to parabolics.  (This definition is similar to one given by Marden in
\cite{OuterCircles}, and is more restrictive than one found in
\cite{CullerShalen}.)  Convergence in $V(C)$ is defined by algebraic
convergence.  We denote the subset of conjugacy classes of minimally
parabolic geometrically finite uniformizations of $C$ by $GF_0(C)
\subseteq V(C)$.  We will give $GF_0(C)$ the algebraic topology.
Marden \cite{Marden} showed that $GF_0(C)$ is open in $V(C)$.
\end{definition}

By \cite{cchs}, conjugacy classes of maximally cusped structures are
dense in the boundary of $GF_0(C)$ in $V(C)$.

Now, we need to recognize indiscrete representations $\rho \co
\pi_1(C) \to \PSL(2,\CC)$.  The following lemma, which is essentially
the Shimizu--Leutbecher lemma \cite[Proposition II.C.5]{maskit},
allows us to do so.  A proof using the notation of this paper can be
found in \cite{CLP:LengthPaper}.

\begin{lemma}\label{lem:MinimalTranslationLength}
Let $\Gamma$ be a discrete torsion free subgroup of $\PSL(2,\CC)$ such
that $M = \HH^3/\Gamma$ has a rank two cusp.  Suppose the point at
$\infty$ projects to the cusp, and $\Gamma_\infty \leq \Gamma$ is the
subgroup of parabolics fixing $\infty$. Then for every $\gamma \in
\Gamma \backslash \Gamma_\infty$ the isometric sphere $I(\gamma)$ has
radius at most $T$, where $T$ is the minimal Euclidean translation
length of all elements of $\Gamma_\infty$.
\end{lemma}

Using the above lemma, we can show the following.  

\begin{lemma}\label{lemma:MaximallyCusped}
For any $\epsilon>0$, there exists a maximally cusped structure on the
$(1, n+1)$--compression body $C$, a system of core tunnels $\tau_1,
\dots, \tau_n$ for $C$, and balls $B_1(\epsilon), \dots,
B_{n-1}(\epsilon)$ each of radius $\epsilon$ such that for $i = 1,
\dots, n-1$, the tunnel $\tau_i$ intersects the ball $B_i(\epsilon)$
in two distinct arcs.  Hence $n-1$ of the $n$ tunnels come within
distance at most $\epsilon$ of self--intersecting.
\end{lemma}

\begin{proof}
Let $\rho_0$ be the geometrically finite representation of the
$(1,n+1)$--compression body constructed in Theorem
\ref{thm:IntersectingGeodesics}, with generators $\alpha, \beta,
\delta_1, \dots, \delta_n$ for $\pi_1(C)$, and where the geodesic
duals to the isometric spheres corresponding to $\rho_0(\delta_1),
\dots, \rho_0(\delta_{n-1})$ glue up to self--intersect.  We need to
recall a bit more detail about where the intersection occurs.  Recall
from the proof of Theorem \ref{thm:IntersectingGeodesics} that in the
universal cover $\HH^3$, for $i=1, \dots, n-1$, the geodesic dual to
the isometric sphere $I(\rho_0(\delta_i))$ intersects the geodesic
running from the center of the isometric sphere $I(\rho_0(\delta_i
\delta_n\inv))$ to the center of the isometric sphere
$I(\rho_0(\delta_n\inv))$, and that these two geodesics have the same
image in $\HH^3/ \rho_0(\pi_1(C))$.

Now, the translation lengths of $\rho_0(\alpha)$ and $\rho_0(\beta)$
are bounded by some number $L$.  We can consider $\rho_0$ as an
element of $V(C)$.  Let $\mathcal{R}$ be the set of all
representations $\rho$ of $\pi_1(C)$ where $\rho(\alpha), \rho(\beta)$
are parabolics fixing infinity with translation length bounded by $L$,
and $\rho(\delta_i) = \rho_0(\delta_i)$.  By suitably normalizing
$\rho(\alpha), \rho(\beta)$ to avoid conjugation, we can view
$\mathcal{R}$ as a subset of $V(C)$.  Note that $\rho_0 \in
\mathcal{R}$.

Also note that for any geometrically finite structure $\rho$ in
$\mathcal{R}$, and for all $i=1, \dots, n-1$, the geodesic arc dual to
the isometric sphere $I({\rho(\delta_i)})$ will intersect the geodesic
running from the center of $I(\rho(\delta_i\delta_n\inv))$ to the
center of $I(\rho(\delta_n\inv))$, and again these geodesics have the
same image in $\HH^3/\rho(\pi_1(C))$, giving self--intersecting
tunnels in each of these structures.

Now, there exists a path in $\mathcal{R}$ from $\rho_0$ to some
indiscrete representation.  Such a path is obtained by decreasing the
minimal translation length of $\rho(\alpha)$ or $\rho(\beta)$ until it
becomes smaller than the radius of an isometric sphere.  This gives an
indiscrete structure by Lemma \ref{lem:MinimalTranslationLength}.
Hence this path intersects $\partial GF_0(C)$ at some point, say
$\rho_\infty$.

Maximally cusped structures are dense in $\partial GF_0(C)$
\cite{cchs}.  Hence there exists a sequence of geometrically finite
representations $\rho_k$ of $\pi_1(C)$ such that the conformal
boundaries of the manifolds $C_k = \HH^3/\rho_k(\pi_1(C))$ are
maximally cusped genus $(n+1)$ surfaces, each $C_k$ is homeomorphic to
the interior of $C$, and the algebraic limit of the $\rho_k$ is
$\rho_\infty$.

Now, for each $\delta_i$, $i=1, \dots, n$, $\rho_k(\delta_i)$
converges to $\rho_\infty(\delta_i)$, hence the center of the
corresponding isometric sphere $I(\rho_k(\delta_i))$ converges to the
center of the isometric sphere $I(\rho_\infty(\delta_i))$.  Similarly,
the centers of the isometric spheres $I(\rho_k(\delta_i\delta_n\inv))$
and $I(\rho_k(\delta_n\inv))$ converge to the centers of the isometric
spheres $I(\rho_\infty(\delta_i\delta_n\inv))$ and
$I(\rho_\infty(\delta_n\inv))$, respectively.

Hence for any $\epsilon>0$, there exists $K>0$ such that if $k>K$, the
geodesic $\tau_k$ dual to $I(\rho_k(\delta_i))$, and its image
$\rho_k(\delta_i\inv\delta_n)(\tau_k)$, with endpoints at the centers
of isometric spheres $I(\rho_k(\delta_i\delta_n\inv))$ and
$I(\rho_k(\delta_n\inv))$, are within distance $\epsilon/2$ of each
other.  Let $p_i$ be a point of distance at most $\epsilon/4$ from
both geodesics.  Note the two geodesics intersect the ball
$B_\epsilon(p_i)$ of radius $\epsilon$ in $\HH^3$.  Let
$B_i(\epsilon)$ denote the image of the ball $B_\epsilon(p_i)$ in the
quotient manifold $\HH^3 /\rho_k(\pi_1(C))$.  The image of the
geodesic $\tau_i$ runs through $B_i(\epsilon)$ in two distinct arcs,
as desired.
\end{proof}

\begin{lemma}\label{lemma:UnfilledPants}
For any $\epsilon>0$ and any integer $n\geq 2$, there exists a finite
volume hyperbolic 3--manifold $M$ with the following properties.
\begin{enumerate}
\item\label{item:1} $M$ is obtained from a manifold $\widehat{M}$
  with genus $(n+1)$ Heegaard surface $S$ by drilling out a collection
  of curves on $S$ corresponding to a full pants decomposition of $S$.
\item\label{item:2} There exists a tunnel system $\tau_1, \dots, \tau_n$ for
  $\widehat{M}$ and balls $B_1(\epsilon), \dots, B_{n-1}(\epsilon)
  \subset M$ such that for each $i=1, \dots, n-1$, the geodesic arc in
  the homotopy class of $\tau_i$ in $M$ intersects $B_i(\epsilon)$ in
  at least two nontrivial arcs.  Hence the arc comes within $\epsilon$
  of self--intersecting in $M$.
\end{enumerate}
\end{lemma}

\begin{proof}
Let $C_D$ be a maximally cusped compression body of Lemma
\ref{lemma:MaximallyCusped}.  The collection of curves on the
conformal boundary of $C_D$ forms a pants decomposition $P$ of the
genus $(n+1)$ surface.

Let $H$ be a genus $(n+1)$ handlebody.  We wish to take a maximally
cusped hyperbolic structure on $H$ for which the rank--1 cusps on
$\partial H$ consist exactly of the curves of $P$.  In fact, there are
infinitely many such structures, which follows as a consequence of
Thurston's Uniformization Theorem (see Morgan
\cite{Morgan:uniformization}).  Let $H_D$ denote one such structure.

Now, consider the convex cores of $C_D$ and of $H_D$.  The boundaries
of the convex cores consist of 3--punctured spheres, which have a
unique hyperbolic structure.  Hence we may glue $\bdy_+ C_D$ to $\bdy
H_D$ via isometry on each 3--punctured sphere, and we obtain a finite
volume hyperbolic 3--manifold $M$ with $(3n+1)$ rank--2 cusps.  One of
these cusps comes from the rank--2 cusp of $C_D$.  The other $3n$ come
from gluing together the $3n$ rank--1 cusps on the boundaries of the
convex cores of $C_D$ and $H_D$.

Note that if we do a trivial Dehn filling of the $3n$ cusps of $M$
that came from rank--1 cusps on $\bdy_+ C_D$ and $\bdy H_D$, then we
obtain a manifold $\widehat{M}$ with a genus $(n+1)$ Heegaard
splitting along a surface we denote $S$.  This gives item
\eqref{item:1}.

The core tunnels $\tau_1, \dots, \tau_n$ of $C_D$ become a system of
tunnels for $\widehat{M}$ under the gluing.  Because the gluing is
by isometry, and the balls $B_i(\epsilon)$ of Lemma
\ref{lemma:MaximallyCusped} lie within the convex core of $C_D$, for
$i=1, \dots, n-1$, the tunnels $\tau_1, \dots, \tau_{n-1}$ satisfy
item \eqref{item:2}.  
\end{proof}

\subsection{Filling the manifold}

In order to obtain a manifold with a tunnel system consisting of arcs
arbitrarily close to self--intersecting, we will take the cusped
manifold of Lemma \ref{lemma:UnfilledPants} and perform Dehn filling
on the $3n$ cusps corresponding to the pants decompositions of the
surface $S$.

\begin{lemma}\label{lemma:DehnTwist}
Let $M$ be the manifold of Lemma \ref{lemma:UnfilledPants}, and let
$S$ be the surface in item \eqref{item:1} of that lemma.  So $M$ is
homeomorphic to the interior of a compact manifold $\overline{M}$ with
torus boundary components $T_1, \dots, T_{3n+1}$, and $S\cap M$
intersects $\overline{M}$ in a surface $\overline{S}$ with boundary on
$T_1, \dots, T_{3n}$.  For each torus boundary component $T_j$,
$j=1, \dots, 3n$, we take a basis for $H_1(T_j)$ consisting of the
curves $\lambda_j$, $\mu_j$, where $\lambda_j$ is a component of
$\partial \overline{S} \cap T_j$, and $\mu_j$ is any curve with
intersection number 1 with $\lambda_j$.  Then Dehn filling
$\overline{M}$ along any slope $\mu_j + k\,\lambda_j$ will yield a
manifold with Heegaard surface $S$.
\end{lemma}

\begin{proof}
Any such slope has intersection number one with the surface $S$.  It
is well known that Dehn filling on such a slope acts as a Dehn twist
on the surface $S$, and the $S$ is a Heegaard surface for every such
Dehn filling (see \cite{lickorish:orientable, rolfsen}).
\end{proof}

We are now ready to prove Theorem \ref{thm:maintheorem} from the
introduction, which we restate.  

\begin{theorem}\label{thm:close-to-intersecting}
For any $\epsilon>0$ and any integer $n\geq 2$, there exists a finite
volume hyperbolic 3--manifold $M$ with a single cusp torus such that
$M$ has the following property.  It admits a system of tunnels $\{
\tau_1, \dots, \tau_n\}$, and a collection of balls $B_1(\epsilon),
\dots, B_{n-1}(\epsilon)$ of radius $\epsilon$ such that for $i=1,
\dots, n-1$, the geodesic arc homotopic to $\tau_i$ intersects
$B_i(\epsilon)$ in two distinct arcs.
\end{theorem}

In other words, the tunnels $\tau_1, \dots, \tau_{n-1}$ have geodesic
representatives which come arbitrarily close to self--intersecting.
Although the proof does not guarantee that these tunnels do
self--intersect, it gives evidence that there exist tunnels which are
not isotopic to geodesics.

\begin{proof}
Let $M$ be the manifold of Lemma \ref{lemma:UnfilledPants}, say with
$\epsilon$ replaced by $\epsilon/4$ in that lemma.  By Lemma
\ref{lemma:DehnTwist}, Dehn filling the cusps of $M$ corresponding to
the pants curves of $\widehat{M}$ along slopes $\mu_j + k_j\,\lambda_j$
will yield a manifold with a genus--$(n+1)$ Heegaard splitting.

By work of Thurston \cite{Thurston}, as $k_j$ approaches infinity, the
Dehn filling approaches the manifold $M$ in the Gromov--Hausdorff
topology.  Even more precisely, work of Brock and Bromberg
\cite{brock-bromberg:density} implies that for any $\epsilon_1>0$, if
$k_j$ is large enough, there is $(1+\epsilon_1)$--bilipschitz
diffeomorphism $\phi$ from the complement of a Margulis tube about the
unfilled cusp to the complement of a Margulis tube about the core of
the filled solid torus.  We may take these Margulis tubes to avoid the
tunnels of our tunnel system.  Moreover, $\phi$ is level preserving in
the unfilled cusp.  Hence for large enough $k_j$, we obtain the
desired result.
\end{proof}

\subsection{Tunnel number ${\mathbf n}$}\label{sec:TunnelNumbern}

It would be nice to add to the conclusions of Theorem
\ref{thm:close-to-intersecting} that $M$ is tunnel number $n$, and not
some lower tunnel number.  Since for a manifold $M$ with one torus
boundary component, a system of $n$ tunnels corresponds to a genus
$n+1$ Heegaard splitting, we can prove that our manifold is tunnel
number $n$ by showing there are no lower genus Heegaard splittings of
$M$.  If the Hempel distance of the Heegaard splitting is high, then
work of Scharlemann and Tomova \cite{ScharlemannTomova} will imply
that there are no lower genus Heegaard splittings.  Hence in this
section we will review Hempel distance and other results on the curve
complex that will allow us to conclude our manifold is tunnel number
$n$.  It should be noted that the final step in this procedure relies
on announced work of Maher and Schleimer \cite{MaherSchleimer}, which
as of yet has not appeared.  Hence we include the result in a separate
section.

\begin{define}\label{def:CurveComplex}
Let $S$ be a closed, oriented, connected surface.  The \emph{curve
complex} $\mathcal{C}(S)$ is the simplicial complex whose vertices are
isotopy classes of essential curves in $S$, and a collection of $k+1$
vertices form a $k$--simplex whenever the corresponding curves can be
realized by disjoint curves on $S$.  For $\alpha$, $\beta$ vertices in
$C(S)$, we define the distance $d_S(\alpha, \beta)$ to be the minimal
number of edges in any path in the 1--skeleton of $\mathcal{C}(S)$
between $\alpha$ and $\beta$. 

The \emph{disk set} of a compression body with outer boundary
homeomorphic to $S$ is defined to consist of vertices in
$\mathcal{C}(S)$ which are the boundaries of essential disks in the
compression body.  Similarly, the disk set of a handlebody with
boundary $S$ consists of vertices of $\mathcal{C}(S)$ which are
boundaries of essential disks in the handlebody.  Note the disk set of
a compression body with outer boundary $S$ is contained in the disk
set of a handlebody with boundary $S$. 
\end{define}

\begin{define}\label{def:hempel}
A Heegaard splitting of a 3--manifold along a surface $S$ has two disk
sets, one on either side of $S$.  The \emph{Hempel distance} of the
Heegaard splitting is defined to be the minimal distance in
$\mathcal{C}(S)$ between those disk sets.  See \cite{hempel:dist}.

More generally, we define the \emph{inner distance} between sets $A$
and $B$ in $\mathcal{C}(S)$ to be
$$d(A,B) = \inf{ d_S(a, b) \mid a\in A, b\in B }.$$
Thus the Hempel distance is the inner distance between the disk sets
on either side of a Heegaard surface.
\end{define}

Scharlemann and Tomova showed that if the Hempel distance of a genus
$g$ Heegaard splitting is strictly greater than $2g$, then the
manifold will have a unique Heegaard splitting of genus $g$ and no
Heegaard splittings of smaller genus \cite{ScharlemannTomova}.

\begin{define}\label{def:handlebody-complex}
For any closed, oriented, connected surface $S$, let the handlebody
graph $\mathcal{H}(S)$ be the graph which has a vertex for each
handlebody with boundary $S$.  Since any handlebody has an associated
disk set, alternately we may think of $\mathcal{H}(S)$ as having
vertices corresponding to disk sets in $\mathcal{C}(S)$.  There is an
edge in $\mathcal{H}(S)$ beween two handlebodies whose disk sets
intersect in the curve complex $\mathcal{C}(S)$.  The distance
$d_H(x,y)$ between any two handlebodies in $\mathcal{H}(S)$ is defined
to be the minimal number of edges in a path between them in
$\mathcal{H}(S)$.
\end{define}

There is a relation $D\co \mathcal{H}(S) \to \mathcal{C}(S)$ defined
as follows.  For a handlebody $V \in \mathcal{H}(S)$, $D(V)$ consists
of the disk set of $V$.  The following lemma was pointed out to us by
S.~Schleimer.

\begin{lemma}\label{lemma:half-lipschitz}
For any handlebodies $V$ and $W$ in $\mathcal{H}(S)$,
$$d_H(V,W) \leq d( D(V), D(W)) + 1.$$
(The right hand side is inner distance in the curve complex.)
\end{lemma}

\begin{proof}
Given handlebodies $V$ and $W$, let $\alpha \in D(V)$ and $\beta \in
D(W)$, with $d_S(\alpha, \beta) = k$.  Take a minimal length path in
$\mathcal{C}(S)$ between $\alpha$ and $\beta$, and denote the vertices
of the path by $\alpha= \alpha_0, \alpha_1, \dots, \alpha_k=\beta$.
Note that for any $i = 1, 2, \dots, k$, there exists a handlebody
which we denote $V_i$ such that $\alpha_{i-1}$ and $\alpha_i$ both
bound essential disks in $V_i$.  Then we obtain a sequence of
handlebodies $V_1, V_2, \dots, V_k$ with $d_H(V_i, V_{i+1}) = 1$.
Since $d_H(V, V_1) = 1 = d_H(V_k, W)$, the sequence of handlebodies
$V, V_1, \dots, V_k, W$ gives a path from $V$ to $W$ of length $k+1$
in $\mathcal{H}(S)$.  Hence $d_H(V,W) \leq d_S(\alpha, \beta) + 1$.
Since $\alpha$ and $\beta$ were arbitrary,
$$d_H(V,W)-1 \leq \inf{ d_S(\alpha,\beta) \mid \alpha \in D(V), \beta
\in D(W)} = d( D(V), D(W)),$$ as desired.
\end{proof}

Maher and Schleimer have proved that $\mathcal{H}(S)$ has infinite
diameter \cite{MaherSchleimer}.  We will use this to prove the
following strengthened version of Theorem
\ref{thm:close-to-intersecting}.

\begin{theorem}\label{thm:Main-TunnelNumN}
For any $\epsilon>0$ and any integer $n\geq 2$, there exists a finite
volume hyperbolic tunnel number $n$ manifold $M$ with a single cusp
torus such that $M$ has the following property.  It admits a system of
tunnels $\{\tau_1, \dots, \tau_n\}$ and a collection of balls
$B_1(\epsilon), \dots , B_{n-1}(\epsilon)$ of radius $\epsilon$ such
that for $i=1, \dots, n-1$, the geodesic arc homotopic to $\tau_i$
intersects $B_i(\epsilon)$ in two distinct arcs.
\end{theorem}

The difference between this theorem and Theorem
\ref{thm:close-to-intersecting} is that here we may conclude that our
manifold is tunnel number $n$, while there we just have a system of
$n$ tunnels.

\begin{proof}
As in the proof of Theorem \ref{thm:close-to-intersecting}, we will
start with the maximally cusped compression body of Lemma
\ref{lemma:MaximallyCusped}, attach to it a maximally cusped
handlebody, and Dehn fill in such a way that the resulting manifold is
geometrically close to the original.  However, we will choose the
maximally cusped structure on our handlebody more carefully, to ensure
that after Dehn filling, the Hempel distance of the resulting Heegaard
splitting remains high.

We first set up notation.  Let $C_0$ be the maximally cusped structure
on the $(1; n+1)$--compression body from Lemma
\ref{lemma:MaximallyCusped}.  Let $S$ denote the positive boundary of
this compression body.  Let $D(C_0)$ denote the disk set of $C_0$ in
$\mathcal{C}(S)$.  The pinched curves on $S$ corresponding to the
rank--1 cusps of $C_0$ form a maximal simplex $P$ in $\mathcal{C}(S)$.

Notice that there is a relation $h$ from disk sets of a compression
body with outer boundary $S$ to the handlebody complex
$\mathcal{H}(S)$, as follows.  For a disk set $D$ of a compression
body, $h(D)$ consists of all $V \in \mathcal{H}(S)$ for which $D$ is a
subset of $D(V)$.  So in particular, $h(D(C_0))$ is a subset of
$\mathcal{H}(S)$.

Recall that a multitwist along $P$ is a collection of Dehn twists, one
along each curve of $P$.  Let $X = \{ h( T(D(C_0))) \mid T \mbox{ is a
  multitwist along } P\}$.  That is, $X$ is the subset of
$\mathcal{H}(S)$ consisting of all handlebodies whose disk sets
contain $T(D(C_0))$ for some multitwist $T$.

Now for the next step of the proof, we show that $X$ has bounded
diameter in $\mathcal{H}(S)$.

First, recall that a Dehn twist along any curve in $P$ is an isometry
of $\mathcal{C}(S)$, fixing $P$ pointwise.  Thus, if we let $K$ denote
the inner distance between $P$ and $D(C_0)$ in $\mathcal{C}(S)$, and
if we let $T$ be any multitwist along $P$, then $T(D(C_0))$ has inner
distance $K$ from $P$.

Now, for any $V, W \in X$, by definition of $X$ there exist
multitwists $T_1$ and $T_2$ such that $T_1(D(C_0)) \subset D(V)$ and
$T_2(D(C_0)) \subset D(W)$.  Then the inner distance $d(D(V), D(W))$
satisfies
\begin{eqnarray*}
d(D(V), D(W)) &\leq & d(D(V), P) + d(P, D(W)) \\
&\leq&  d(T_1(D(C_0)), P) + d(P, T_2(D(C_0))) = 2K.
\end{eqnarray*}

By Lemma \ref{lemma:half-lipschitz}, $d_H(V,W) \leq 2K+1$.  Hence $X$
has bounded diameter in $\mathcal{H}(S)$.

We are finally ready to choose the maximally cusped structure on our
handlebody.  Since the diameter of $\mathcal{H}(S)$ is infinite
\cite{MaherSchleimer}, we may choose a handlebody $Y$ in
$\mathcal{H}(S)$ such that $\inf{ d_H(Y, V) \mid V \in X } = N$, where
$N$ is some number, at least $2n+3$.  Let $H_0$ be a maximally cusped
structure on $Y$ with the curves $P$ on $\partial Y$ pinched to
rank--1 cusps.  The fact that such a structure exists follows as a
consequence of Thurston's Uniformization Theorem (see Morgan
\cite{Morgan:uniformization}). 

As in the proof of Theorem \ref{thm:close-to-intersecting}, glue $H_0$
to $C_0$ by isometry, and then Dehn fill $P$ along slopes of the form
$\mu_i + k \,\lambda_i$.  Any such Dehn filling along $P$ fixes the
disk set of $Y$, and modifies the disk set of $C_0$ by applying
multitwist along $P$.  By choice of $Y$ and Lemma
\ref{lemma:half-lipschitz}, any such Dehn filling will yield a
manifold with large Hempel distance, larger than $N-1>2n+2$.  Then
work of Scharlemann and Tomova implies that the minimal genus Heegaard
splitting must have genus at least $(n+1)$, which means that the
manifold is tunnel number $n$.

On the other hand, the same proof as that of Theorem
\ref{thm:close-to-intersecting} applies to show that $n-1$ tunnels are
arbitrarily close to self--intersecting.
\end{proof}

\bibliographystyle{hamsplain}
\bibliography{references}
\end{document}